\documentclass[11 pt]{amsart}
\usepackage{graphicx}

\usepackage{amsmath}
\usepackage{mathtools}
\usepackage
{amstext, amscd, setspace, tikz-cd,mathtools, enumerate, graphics, latexsym, siunitx}

\usepackage{pst-node}
\usepackage{tikz-cd}

\usepackage{amssymb}
\usepackage{hyperref}
\hypersetup{colorlinks=true,
    linkcolor=blue,
    filecolor=magenta,      
    urlcolor=cyan,}

\usepackage{PTSerif}
\usepackage{graphicx}

\usepackage[cmtip,all]{xy}
\newcommand{\properideal}{%
\mathrel{\ooalign{$\lneq$\cr\raise.22ex\hbox{$\lhd$}\cr}}}

\newcommand{\q}{\mathfrak{q}}
\newcommand{\p}{\mathfrak{p}}
\newcommand{\D}{\mathcal{D}}
\newcommand{\m}{\mathfrak{m}}

\newcommand{\Z}{\mathbb{Z}}

\newcommand{\Ass}{\operatorname{Ass}}

\newcommand{\Supp}{\operatorname{Supp}}

\newcommand{\Hom}{\operatorname{Hom}}

\newcommand{\Ext}{\operatorname{Ext}}

\theoremstyle{plain}
\newtheorem{theorem}{Theorem}[section]
\newtheorem{lemma}[theorem]{Lemma}

\newtheorem{corollary}[theorem]{Corollary}
\newtheorem{proposition}[theorem]{Proposition}
\newtheorem{remark}[theorem]{Remark}
\newtheorem{definition}[theorem]{Definition}

\newtheorem{hypothesis}[theorem]{Hypothesis}
\newtheorem{setup}[theorem]{Setup}
\newtheorem{point}[theorem]{}

\title[On Bass numbers of graded components of local cohomology modules]{On Bass numbers of graded components of local cohomology modules supported on $\mathfrak{C}$-monomial ideals in mixed characteristic}

\author{\textsc{Sayed Sadiqul Islam}}
\address{Department of Mathematics, IIT Bombay, Powai, Mumbai 400076, India}
\email{ssislam1997@gmail.com, 22d0786@iitb.ac.in}
\author{\textsc{Tony J. Puthenpurakal}}
\address{Department of Mathematics, IIT Bombay, Powai, Mumbai 400076, India}
\email{tputhen@math.iitb.ac.in}

\date{\today}

\begin{document}

\begin{abstract}
 Let $A$ be a Dedekind domain of characteristic zero such that for each height one prime ideal $\p$ in $A$, the local ring $A_{\p}$ has mixed characteristic with finite residue field. Suppose that $R=A[X_1,\ldots,X_n]$ is a standard $\mathbb{N}^n$-graded polynomial ring over $A$, i.e., $\deg A=\underline{0}\in \mathbb{N}^n$ and $\deg(X_j)=e_j\in \mathbb{N}^n$. Let $I$ be a $\mathfrak{C}$-monomial ideal of $R$ and let $M:= H^i_I(R)=\bigoplus_{\underline{u}\in \mathbb{Z}^n}M_{\underline{u}}$. Recently, the second author and S. Roy  [2025, J. Algebra 681, 1–21] proved that for a fixed $\underline{u}\in\mathbb{Z}^n$, the Bass numbers $\mu_i(\p,M_{\underline{u}})$ are finite for each prime ideal $\p$ in $A$ and for every $i\geq 0$. Let  for a subset of $U$ of $\mathcal{S}=\{1, \ldots, n\}$, define a block to be the set
$\displaystyle\mathcal{B}(U)=\{\underline{u} \in \Z^n \mid u_i \geq 0 \mbox{ if } i \in U \mbox{ and } u_i \leq -1 \mbox{ if } i \notin U \}$. Note that $\bigcup_{U\subseteq \mathcal{S}}\mathcal{B}(U)=\mathbb{Z}^n$. In this article, the main result we establish is that for a fixed prime ideal $\p$ in $A$ and $i\geq 0$, the set of Bass numbers $\{\mu_i(\p,M_{\underline{u}})\mid \underline{u}\in \mathbb{Z}^n\}$ is constant 
on $\mathcal{B}(U)$ for each subset $U$ of $\{1, \ldots, n\}$. Our idea is to prove this by carrying out a comprehensive study of the structure theorem for the graded components of $M$ when $A$ is a complete DVR of mixed characteristic with finite residue field.
\end{abstract}

\maketitle

\section{Introduction}
Introduced in the 1960s as a tool for studying sheaves and their cohomology in algebraic geometry, local cohomology has since found extensive applications in commutative algebra. Over the years, it has developed into a fundamental tool and a vibrant area of research in its own right. 

Local cohomology modules are not finitely generated in general. The lack of finite generation has turned attention of researchers to study other type of finiteness properties; see, for instance, \cite{HS-93, Lyu-93, Lyu-97, Lyu-2000}. Nevertheless, for the purposes of this article, we restrict our attention to the graded components of local cohomology modules. We now review some related literature and also state a result of the second author, joint with a coauthor, from \cite{TS-25}, which motivates the present work.

Let $S=\oplus_{n\geq 0}S_n$ be standard graded Noetherian ring and $S_+$ be irrelevant ideal. The theory of local cohomology  with respect to $S_+$ is particularly well behaved \cite[Theorem 15.1.5]{BS-13}. It is well-known that if $M$ is finitely generated graded $S$-module then for $i\geq 0$,
\begin{enumerate}
    \item $H^i_{S_+}(M)_n$ is finitely generated $S_0$-module for all $n\in \mathbb{Z}$.
    \item $H^i_{S_+}(M)_n=0$ for all $n\gg 0$.
\end{enumerate}
This fact motivated the second author to study the graded components of $H^i_I(R)$ where  $R$ is a standard graded polynomial ring $A[X_1,\ldots,X_n]$ over a regular ring  $A$  containing a field of characteristic zero and $I$ is an arbitrary homogeneous ideal of $R$ (see, \cite{TP-collect}).  In \cite{TS-24}, he, together with the coauthor, established some of these results under weaker assumptions on the base ring $A$. Later they studied the case when $A=B^G$, the ring of invariants of a regular domain $B$ that contains a field $K$ of characteristic zero and $G$ is a finite subgroup of the automorphism group of $B$ (see, \cite{TS-19}). 

Let $R=K[X_1,\ldots,X_n]$ be a standard $\mathbb{N}^n$-graded polynomial ring over a field $K$ and let $I$ be a monomial ideal of $R$. Then, $H^i_I(R)$ is endowed with natural $\mathbb{Z}^n$-graded structure. In addition, it has the structure of a $\D$-module. There has been lot research on understanding these structures on $H^i_I(R)$. In fact, $\D$-module approach and $\mathbb{Z}^n$-graded approach to the study of local cohomology modules are equivalent. We refer the reader to the beautiful survey article \cite{Mon-13} by Montaner for a more detailed discussion.
\begin{hypothesis}\label{hypo-TS}
Let $A$ be a Dedekind domain of characteristic zero such that for each height one prime ideal $\p$ in $A$, the local ring $A_{\p}$ has mixed characteristic with finite residue field. Suppose that $R=A[X_1,\ldots,X_n]$ is a standard $\mathbb{N}^n$-graded polynomial ring over $A$, i.e., $\deg A=\underline{0}\in \mathbb{N}^n$ and $\deg(X_j)=e_j\in \mathbb{N}^n$. Let $I$ be a $\mathfrak{C}$-monomial ideal of $R$ and let $M:= H^i_I(R)=\bigoplus_{\underline{u}\in \mathbb{Z}^n}M_{\underline{u}}$. Assume that $\mathfrak{Z}:= \Ass_A H^i_I(R)\setminus \{0\}$.
\end{hypothesis}
Recall that an ideal $I$ of $R$ is called $\mathfrak{C}$-monomial ideal if it can be generated by elements of the form $aU$ where $a\in A$ (not necessarily a unit) and $U$ is a monomial in the variables $X_1,\ldots,X_n$. We note that if $I$ is a $\mathfrak{C}$-monomial ideal of $R$, then the components of $H^i_I(R)$ may not be finitely generated as $A$-module; see \cite[Example 7.1]{TS-25}.

Recently, the second author together with the coauthor in \cite{TS-25} proved the following result concerning the structure theorem for the multigraded components of $M$.
\begin{theorem}\cite[Theorem 1.3]{TS-25}\label{structure theorem due to TS}
    Assume the hypothesis as in \ref{hypo-TS}. Take any prime ideal $\q$ of $A$. We set $T=\widehat A_{\q}$ and $N=M\otimes_A T=H^i_{IS}(S)$, where $S=T[X_1,\ldots,X_n]$. Let $K$ and $K_{\q}$ be the quotient fields of $A$ and $T$ respectively. For a fixed $\underline{u}\in \mathbb{Z}^n$, the following holds:
    \begin{enumerate}[\rm (1)]
        \item $t(M_{\underline{u}})=\bigoplus_{\p\in \mathfrak{Z}}\Gamma_{\p}(M_{\underline{u}})$, where $t(M_{\underline{u}})$ denote the torsion $A$-submodule of $M_{\underline{u}}$.
        \item For $\p\in\mathfrak{Z}$, there are finite numbers $l({\underline{u}}), \beta_j({\underline{u}}), t({\underline{u}})$ such that $$\Gamma_{\p}(M_{\underline{u}})=E_A\left(\frac{A}{\p}\right)^{l({\underline{u}})} \oplus\left(\bigoplus_{j=1}^{t({\underline{u}})}\frac{A}{\p^{\beta_j({\underline{u}})}A}\right).$$
        \item Let $\overline{N_{\underline{u}}}:= \frac{N_{\underline{u}}}{\Gamma_{\q}(N_{\underline{u}})}$. Then $\overline{N_{\underline{u}}}=T^{a(\underline{u})}\oplus K_{\q}^{b(\underline{u})}$ for some finite numbers $a(\underline{u}),b(\underline{u})\geq 0$.
        \item $M_{\underline{u}}\otimes_A K=K^{\alpha(\underline{u})}$ for some $\alpha(\underline{u})\geq 0$. Moreover, $\{\alpha(\underline{u})\mid \underline{u}\in \mathbb{Z}^n\}$ is a finite set.
    \end{enumerate}
\end{theorem}
It was further shown that if $A$ is a PID then each component can be written as a direct sum of its torsion part and torsion-free part. 

Let $K$ be a field of characteristic zero and let $R=A[X_1,\ldots,X_n]$ be a standard $\mathbb{N}^n$-graded polynomial ring over $A$ where $A=K[[Y]]$ is a formal power series ring in one variable. The authors proved an analogue of the structure theorem for components of $H^i_I(R)$ in this case (see, \cite[Theorem 7.1]{TS-23}). Let $\mathcal{B}(U)$ denote a block (see, \ref{blocks} below for definition) corresponding to a subset $U$ of $\{1,\ldots,n\}$. It is also proved that for a fixed prime ideal $\p$ in $A$ and $i\geq 0$, the set of Bass numbers $\{\mu_i(\p,M_{\underline{u}})\mid \underline{u}\in \mathbb{Z}^n\}$ is constant 
on $\mathcal{B}(U)$ for each subset $U$ of $\{1, \ldots, n\}$ (see, \cite[Theorem 5.6]{TS-23}).

 Consider the following setup.
\begin{setup}\label{main setup of our paper}
    Let $(A,\pi A,F)$ be a complete DVR of mixed characteristic $(0,p)$ and with finite residue field.  Suppose that $R=A[X_1,\ldots,X_n]$ is a standard $\mathbb{N}^n$-graded polynomial ring over $A$, i.e., $\deg A=\underline{0}\in \mathbb{N}^n$ and $\deg(X_j)=e_j\in \mathbb{N}^n$. Let $I$ be a $\mathfrak{C}$-monomial ideal of $R$ and let $M:= H^i_I(R)=\bigoplus_{\underline{u}\in \mathbb{Z}^n}M_{\underline{u}}$. By Theorem \ref{structure theorem due to TS},  there are finite numbers $l({\underline{u}}), \beta_j({\underline{u}}), t({\underline{u}}), a(\underline{u}),b(\underline{u})\geq 0$ (depending on   $\underline{u}$) such that
$$\Gamma_{\pi}(M_{\underline{u}})=E_A\left(\frac{A}{\pi A}\right)^{l({\underline{u}})} \oplus\left(\bigoplus_{j=1}^{t({\underline{u}})}\frac{A}{\pi^{\beta_j({\underline{u}})}A}\right)$$ and $$\overline{M_{\underline{u}}}:= \frac{M_{\underline{u}}}{\Gamma_{\pi}(M_{\underline{u}})}=A^{a(\underline{u})}\oplus K^{b(\underline{u})}$$
where $K$ is the quotient field of $A$.
\end{setup} 
\medskip
Let  for subset of $U$ of $\mathcal{S}=\{1, \ldots, n\}$, define a block to be the set
$\mathcal{B}(U)=\{\underline{u} \in \Z^n \mid u_i \geq 0 \mbox{ if } i \in U \mbox{ and } u_i \leq -1 \mbox{ if } i \notin U \}$. Note that $\bigcup_{U\subseteq \mathcal{S}}\mathcal{B}(U)=\mathbb{Z}^n$. We now state the results established in this paper. The first result among them is the following.
\medskip

\noindent
\textbf{Theorem A.} (Theorem \ref{proof of theorem a})
\phantomsection \label{Theorem A}
\textit{ Assume the hypothesis as in the Setup \ref{main setup of our paper}.
Then, the sets  $\{a(\underline{u})\mid \underline{u}\in \mathbb{Z}^n\}$ and $\{b(\underline{u})\mid \underline{u}\in \mathbb{Z}^n\}$ are constant on $\mathcal{B}(U)$ for each subset $U$ of $\{1, \ldots, n\}$.
}
\medskip

The next result we prove is the following.
\medskip

\noindent
\textbf{Theorem B.} (Theorem \ref{proof of theorem b})
\phantomsection \label{Theorem B}
\textit{ Assume the hypothesis as in the Setup \ref{main setup of our paper}. 
Then, the sets  $\{l(\underline{u})\mid \underline{u}\in \mathbb{Z}^n\}$ and $\{t(\underline{u})\mid \underline{u}\in \mathbb{Z}^n\}$ are constant on $\mathcal{B}(U)$ for each subset $U$ of $\{1, \ldots, n\}$.
}
\bigskip

Since  $[-1,1]^n\cap \mathbb{Z}^n$ is a finite set, we can define a finite number as $m:=\max\{\beta_j({\underline{u}})\mid 1\leq j\leq t({\underline{u}})\  \text{where} \ \underline{u}\in [-1,1]^n\cap\  \mathbb{Z}^n\}$. The following result proves that the set $\{\beta_j({\underline{u}})\mid 1\leq j\leq t({\underline{u}})\  \text{where} \ \underline{u}\in  \mathbb{Z}^n\}$ is bounded above by $m$.
\medskip

\noindent
\textbf{Theorem C.} (Theorem \ref{proof of proposition c})
\phantomsection \label{Theorem C}
\textit{ Let the hypothesis be as in Setup \ref{main setup of our paper} and $m$ be as in last paragraph. Then, the set $\{\beta_j({\underline{u}})\mid 1\leq j\leq t({\underline{u}})\  \text{where} \ \underline{u}\in  \mathbb{Z}^n\}$ is finite. In fact,  $\beta_j({\underline{u}})\leq m$ for all $1\leq j\leq t({\underline{u}})$ where  $\underline{u}\in  \mathbb{Z}^n$.
}

\bigskip
\begin{setup}\label{second setup}
Let $s=\max\{\beta_j({\underline{u}})\mid 1\leq j\leq t({\underline{u}})\  \text{where} \ \underline{u}\in  \mathbb{Z}^n\}$. The above theorem says that $s$ is a finite number and $s\leq m$. Let the hypothesis be as in Setup \ref{main setup of our paper}. Then there are numbers $\alpha_j(\underline{u})$ such that 
$$\Gamma_{\pi}(M_{\underline{u}})=E_A\left(\frac{A}{\pi A}\right)^{l({\underline{u}})} \oplus\left(\bigoplus_{j=1}^{s}\left(\frac{A}{\pi^jA}\right)^{\alpha_j(\underline{u})}\right).$$
\end{setup}
We should remark that in the above expression of $\Gamma_{\pi}(M_{\underline{u}})$, it is possible that $\alpha_j(\underline{u})=0$ for some $1\leq j\leq s$.

\medskip

\noindent
\textbf{Theorem D.} (Theorem \ref{proof of theorem D})
\phantomsection \label{Theorem D}
\textit{ Let the hypothesis be as in Setup \ref{second setup}. Let $1\leq j\leq s$. Then, the set $\{\alpha_j(\underline{u})\mid \underline{u}\in \mathbb{Z}^n\}$ is constant on $\mathcal{B}(U)$ for each subset $U$ of $\{1, \ldots, n\}$.
}

\bigskip
Assume the hypothesis as in  \ref{hypo-TS}. As mentioned earlier, the components of the local cohomology modules $H^i_I(R)$ need not be finitely generated. Consequently, one cannot generally expect these components to have finite Bass numbers. Nevertheless, somewhat unexpectedly, the authors  in \cite[Corollary 6.6]{TS-25} proved that for a fixed $\underline{u}\in\mathbb{Z}^n$, the Bass numbers $\mu_i(\p,M_{\underline{u}})$ are finite for each prime ideal $\p$ in $A$ and for every $i\geq 0$. As an application of our results, we prove the following surprising fact, which is the main and final result of this paper.

\medskip

\noindent
\textbf{Theorem E.} (Theorem \ref{Bass numbers are constant} )
\phantomsection \label{Theorem E}
\textit{Assume the hypothesis as in  \ref{hypo-TS}.  Fix a  prime ideal $\p\subseteq A$ and  $i\geq 0$. Then, the set of Bass numbers $\{\mu_i(\p,M_{\underline{u}})\mid \underline{u}\in \mathbb{Z}^n\}$ is constant 
on $\mathcal{B}(U)$ for each subset $U$ of $\{1, \ldots, n\}$.}
\medskip

We now describe in brief the contents of this paper. The paper is organized in seven sections. Section 2 introduces the necessary preliminaries. In Section 3, we prove some auxiliary results that are needed throughout the paper. Section 4 is devoted to the proof of \hyperref[Theorem A]{Theorem A}. In section 5, we prove \hyperref[Theorem B]{Theorem B}. In Section 6, we prove \hyperref[Theorem C]{Theorem C} and \hyperref[Theorem D]{Theorem D}. Finally, the last section proves  \hyperref[Theorem E]{Theorem E}.
\section{Preliminaries}
In this section, we review the necessary preliminaries and recall several well-known results from the literature. More specifically, we recall the theory of Eulerian modules introduced by Ma and Zhang \cite{MaZh-14}, as well as the theory of straight modules due to Yanagawa \cite{Yanagawa-01}.

Let $S=K[X_1,\ldots,X_n]$ be a standard $\mathbb{N}$-graded polynomial ring over a field $K$ of arbitrary characteristic i.e., $\deg X_i=1$, and $\deg a=0$ for all $X_i$ and nonzero $a\in K$. Denote by $\m$, the homogeneous maximal ideal $\m=(X_1,\ldots,X_n)$ of $S$. Let $\D$ be the corresponding ring of differential operators on $S$ i.e., $\D=S\left\langle\partial_i^{[j]} \mid 1 \leq i \leq n,  1 \leq j\right\rangle$ where $\partial_i^{[j]}=\frac{1}{j!}\frac{\partial^j}{\partial x_i^j}$. Note that $\D$ has natural $\mathbb{Z}$-grading given by $\deg X_i=1$, $\deg \partial_i^{[j]}=-j$, and $\deg c=0$ for each $X_i$, $\partial_i^{[j]}$ and each nonzero $c\in K$. Now we define a $\mathbb{Z}$-graded Eulerian $\D$-module which was introduced by Ma and Zhang in \cite{MaZh-14}.
 \begin{definition}
      A $\mathbb{Z}$-graded $\D$-module $M$ is called  Eulerian if for every $r\geq 1$, each homogeneous element $m\in M$ satisfies $$E_r\cdot  m=\binom{\deg m}{r}m$$ where $$E_r=\sum_{i_1+i_2+\cdots+i_n=r, i_1 \geq 0, \ldots, i_n \geq 0} X_1^{i_1} \cdots X_n^{i_n} \partial_1^{\left[i_1\right]} \cdots \partial_n^{\left[i_n\right]}.$$
 \end{definition}
 It is shown in \cite{MaZh-14} that the polynomial ring $S$ is Eulerian and for homogeneous ideals $I_1,\ldots,I_t$ of $S$, $H^{i_1}_{I_1}\ldots H^{i_t}_{I_t}(S)$ is also Eulerian as $\D$-module. In fact, the following result is of particular importance to us.
 \begin{theorem}\cite[Theorem 5.6]{MaZh-14}\label{theorem of ma and zhang direct sum}
     Let $M$ be an  $\mathbb{Z}$-graded Eulerian $\D$-module. If $\Supp_RM=\{\m\}$, then $M$ as a graded $\D$-module is a sum of \ $^*E(n)$ where \ $^*E$ is the graded injective hull of $S/\m$.
 \end{theorem} 
Now we consider an $\mathbb{N}^n$-grading on $S$ such that $S=\bigoplus_{\underline{u}\in \mathbb{N}^n}S_{\underline{u}}= \bigoplus_{\underline{u}\in \mathbb{N}^n}\underline{X}^{\underline{u}}$, where $\underline{X}^{\underline{u}}=\prod_{i=1}^n X_i^{u_i}$ is the monomial with exponent ${\underline{u}\in \mathbb{N}^n}$. For $\underline{u}=(u_1,\ldots,u_n)\in \mathbb{Z}^n$, let $\operatorname{supp}_{+}(\underline{u}):=\{i\mid u_i>0\}$. By $\deg y$, we denote the homogeneous degree of an element $y$. The corresponding ring of differential operators $\D$ of $S$ is naturally $\mathbb{Z}^n$-graded where $\deg(X_i)=e_i$ for all $i$, $\deg \partial_i^{[j]}=-e_j$ for all $j$, and $\deg a=\underline{0}$ for all nonzero $a\in K$.
\begin{definition}
    For $i,r\geq 1$, define $E^i_r :=X_i^r\  \partial_i^{[r]}$. A $\mathbb{Z}^n$-graded $\D$-module $M$ is called a Eulerian if for every $i\geq 1$, each  element $y\in M$ with $\deg y=\underline{u}$ satisfies
    $$E^r_i\cdot y=\binom{u_i}{r}\cdot y$$
  for all $r\geq 1$.
\end{definition}
\begin{lemma}
    The polynomial ring $S$  itself is a $\mathbb{Z}^n$-graded Eulerian $\D$-module.
\end{lemma}
\begin{proof}
    It is enough to prove for monomials of the form $X_1^{j_1}X_2^{j_2}\ldots X_n^{j_n} $. Let this monomial be $f$. Then $\deg f=(j_1,j_2,\ldots,j_n)$. Fix a $i\geq 1$.
    Then,
    \begin{align*}
        E^r_i\cdot f\ &=X_i^r\  \partial_i^{[r]} (X_1^{j_1}X_2^{j_2}\ldots X_n^{j_n})\\&= X_1^{j_1}\ldots X_{i-1}^{j_{i-1}}X_{i+1}^{j_{i+1}}\ldots X_n^{j_n} X_i^r\  \partial_i^{[r]}(X_i^{j_i})\\&=X_1^{j_1}\ldots X_{i-1}^{j_{i-1}}X_{i+1}^{j_{i+1}}\ldots X_n^{j_n} \binom{j_i}{r}X_i^{j_i}\\&= \binom{j_i}{r} f
    \end{align*} 
    for all $r\geq 1$.
\end{proof}
The following proposition shows that if a module is $\mathbb{Z}^n$-graded Eulerian $\D$-module, this property is preserved when we localize the module at a homogeneous element.
\begin{proposition}\label{localization is eulerian}
    If $M$ is is a $\mathbb{Z}^n$-graded Eulerian $\D$-module, so is $T^{-1}M$ for each homogeneous multiplicative system $T\subset S$. In particular, $M_f$ for each homogeneous element $f\in S$ is a $\mathbb{Z}^n$-graded Eulerian $\D$-module.
\end{proposition}
\begin{proof}
    We note that $\partial_i^{[r]}$ is $S^{p^e}$-linear if $p^e \geq r+1$. So we have
$$
\partial_i^{[r]}(m)=\partial_i^{[r]}\left(f^{p^e} \cdot \frac{m}{f^{p^e}}\right)=f^{p^e} \cdot \partial_i^{[r]}\left(\frac{m}{f^{p^e}}\right).
$$

So if $p^e \geq r+1$ and $f \in T$, then $\partial_i^{[r]}\left(\frac{m}{f^{p^e}}\right)=\frac{1}{f^{p^e}} \partial_i^{[r]}(m)$  in $T^{-1} M$. Therefore,
$$
E^r_i \cdot \frac{m}{f^{p^e}}=\frac{1}{f^{p^e}} E^r_i \cdot m.
$$
Take $\frac{m}{f^t} \in T^{-1} M$ where $\deg(m)=(m_1,\ldots,m_n)$ and $\deg(f)=(r_1,\ldots,r_n)$. Then $\deg(\frac{m}{f^t})=(m_1-tr_1,\ldots,m_n-tr_n)$. Multiplying both the numerator and denominator by a large power of $f$, we write $\frac{m}{f^t}=\frac{f^{p^e-t} m}{f^{p^e}}$ for some $p^e \geq \max \{r+1, t\}$. Now
\begin{align*}
    E^r_i \left( \frac{m}{f^t} \right) & =  E^r_i \left( \frac{f^{p^e - t} m}{f^{p^e}} \right) \\
    & = \frac{1}{f^{p^e}} E^r_i ( f^{p^e - t} m ).
\end{align*}
 We note that \begin{align*}
     \deg( f^{p^e - t} m) &=\deg( f^{p^e - t})+\deg(m)\\&=p^e\deg(f)-\deg(f^t)+\deg(m)\\&= p^e(r_1,\ldots,r_n)-t(r_1,\ldots,r_n)+(m_1,\ldots,m_n)\\&=(p^er_1-tr_1+m_1,\ldots,p^er_n-tr_n+m_n).
 \end{align*} 
 Hence
 \begin{align*}
     E^r_i \left( \frac{m}{f^t} \right) &=\frac{1}{f^{p^e}}\binom{p^er_i+m_i-tr_i}{r}f^{p^e - t} m\\&= \binom{m_i-tr_i}{r} \frac{m}{f^t}.
  \end{align*}
\end{proof} 
Also if $M$ is a $\mathbb{Z}^n$-graded Eulerian $\D$-module, so each $\mathbb{Z}^n$-graded submodule and $\mathbb{Z}^n$-graded quotients of $M$ (proof is similar to \cite[Proposition 2.8]{MaZh-14}).
\begin{remark}\label{remark of multigraded eulerian}\normalfont
  The natural localization map $S\rightarrow S_f$ is a morphism of $\D$-modules. Let $I\subseteq S$ be a homogeneous ideal of $S$ generated by homogeneous sequence of elements $\underline f= f_1,f_2,\ldots,f_r\in S$. Then the \v Cech complex of $S $ with respect to  $\underline f$ is defined by 
  \begin{equation}\label{equation of cech complex}
      \Check{C}^{\bullet}(\underline f,S):\  0\rightarrow S\rightarrow \bigoplus_iS_{f_i}\rightarrow \bigoplus_{i,j}S_{f_if_j}\rightarrow\ldots \rightarrow S_{f_1\ldots f_r}\rightarrow 0
  \end{equation}
		where the maps on every summand are localization map up to a sign. The local cohomology module of $S$ with support on $I$  is defined by 
		$$H_I^i(S)=H^i( \Check{C}^{\bullet}(\underline f,S)).$$
		Since $S$ is $\mathbb{Z}^n$-graded Eulerian as a $\D$-module, so is each localization in equation \ref{equation of cech complex} by Proposition \ref{localization is eulerian}. Therefore, $H^i_I(S)$ is $\mathbb{Z}^n$-graded Eulerian as $\D$-module.
\end{remark}

Now we recall the definition of straight module which was introduced by Yanagawa in \cite{Yanagawa-01}.
\begin{definition}
    A $\mathbb{Z}^n$-graded $S$-module $M=\bigoplus_{\underline{u}\in \mathbb{Z}^n}M_{\underline{u}}$ is called straight, if the following two conditions are satisfied.
    \begin{enumerate}[\rm (a)]
        \item$ \dim_K M_{\underline{u}}<\infty$ for all $\underline{u}\in \mathbb{Z}^n$.
        \item The multiplication map $M_{\underline{u}}\ni Y\rightarrow \underline{X}^{\underline{v}} Y\in M_{\underline{u}+\underline{v}} $ is bijective for all $\underline{u}\in \mathbb{Z}^n$ and  $\underline{v}\in \mathbb{N}^n$ with $\operatorname{supp}_{+}(\underline{u}+\underline{v})= \operatorname{supp}_{+}(\underline{u})$.
    \end{enumerate}
\end{definition}
 Musta\c t\v a \cite{Mustata-20} and Terai \cite{Terai-98} proved the following result which is crucial to prove many of our results.
 \begin{theorem}\cite{Mustata-20, Terai-98}\label{theorem of terai and mustata}
   Let  $I_{\scriptscriptstyle \triangle}$
 be a squarefree monomial ideal. For all $i\geq 0$, the local cohomology module $H^i_{I_{\scriptscriptstyle \triangle}}(\omega_S)=H^i_{I_{\scriptscriptstyle \triangle}}(S)(-1,\ldots,-1)$ is a straight module.
 \end{theorem}
We use the following descriptions of blocks as given in \cite{TS-23}, and we adopt the same terminology throughout the paper. 

 \begin{point} \normalfont\textbf{Blocks:}\label{blocks}
     Let $\mathcal{S}$ denote the set $\{1, \ldots, n\}$ and $U$ be a subset (may be empty) of $\mathcal{S}$. 

\noindent
\begin{minipage}{.7\textwidth}
\hspace{0.25cm} We define a block to be
$$\mathcal{B}(U)=\{\underline{u} \in \Z^n \mid u_i \geq 0 \mbox{ if } i \in U \mbox{ and } u_i \leq -1 \mbox{ if } i \notin U \}.$$

Let $n=2$.
In the figure on the right, $\mathcal{B}({\{1,2\}}), \mathcal{B}({\{2\}}), \mathcal{B}({\phi}),$ and $\mathcal{B}({\{1\}})$ are presented by the shaded parts $B_1, B_2, B_3,$ and $B_4$ respectively.
\end{minipage}
\begin{minipage}{.25\textwidth}
\begin{center}
	\begin{tikzpicture}[scale=0.15]
	\draw[->] (-8.5,0)--(9,0) node[right]{$u$};
	\draw[->] (0,-8.5)--(0,9) node[above]{$v$};
	\draw[densely dotted, red] (-2,8.5)--(-2,-8.5);
	\draw[densely dotted, red] (-8.5,-2)--(8.5, -2);
	\draw[fill=red,fill opacity=0.35,draw=none] (0,8.5)--(0,0)--(8.5,0)--(8.5,8.5)--(0,8.5);
	\draw[fill=red,fill opacity=0.35,draw=none] (-2,8.5)--(-2,0)--(-8.5,0)--(-8.5,8.5)--(-2,8.5);
	\draw[fill=red,fill opacity=0.35,draw=none] (-2,-8.5)--(-2,-2)--(-8.5,-2)--(-8.5,-8.5)--(-2,-8.5);
	\draw[fill=red,fill opacity=0.35,draw=none] (0,-8.5)--(0,-2)--(8.5,-2)--(8.5,-8.5)--(0,-8.5);
	\node[blue,draw=none] at (4,-6) {$B_4$};
	\node at (0,-2){\textcolor{red}{$\bullet$}};
	\node[blue,draw=none] at (4,4) {$B_1$};
	\node at (0,0){\textcolor{red}{$\bullet$}};
	\node[blue,draw=none] at (-6,4) {$B_2$};
	\node at (-2,0){\textcolor{red}{$\bullet$}};
	\node[blue,draw=none] at (-6,-6) {$B_3$};
	\node at (-2,-2){\textcolor{red}{$\bullet$}};
    \node[blue,draw=none] at (3,-3) {\tiny $(0,-1)$};
	\node[blue,draw=none] at (2.5,1) {\tiny $(0,0)$};
	\node[blue,draw=none] at (-5,1) {\tiny $(-1,0)$};
	\node[blue,draw=none] at (-5.5,-3) {\tiny $(-1,-1)$};
	\node at (0,-11.5) {\textit{Figure: Blocks}};
	\end{tikzpicture}
	\end{center}
\end{minipage}
\end{point}
\section{Some preparatory results}
In this section, we prove some auxiliary results that are used throughout the proofs of our results. We begin with defining quasi-straight module over a polynomial ring with coefficients in an arbitrary Noetherian ring.
\begin{definition}
    Let $S=T[X_1,\ldots,X_n]$ be a standard $\mathbb{N}^n$-graded polynomial ring over an arbitrary Noetherian ring $T$. We say that a $\mathbb{Z}^n$-graded $S$-module $M=\bigoplus_{\underline{u}\in \mathbb{Z}^n}M_{\underline{u}}$ is a quasi-straight module if $M_{\underline{u}}\cong M_{\underline{u}+\underline{v}}$ for all $\underline{u}\in \mathbb{Z}^n$ and  $\underline{v}\in \mathbb{N}^n$ with $\operatorname{supp}_{+}(\underline{u}+\underline{v})= \operatorname{supp}_{+}(\underline{u})$.
\end{definition}
It is clear that if $T$ is a field, then a straight module over $S$ is quasi-straight. For a $\mathbb{Z}^n$-graded module $M=\bigoplus_{\underline{u}\in \mathbb{Z}^n}M_{\underline{u}}$, let $M(\underline{-1})$ denote the module $M(-1,\ldots,-1)$. Let $\mathcal{B}(U)$ denote a block (see, \ref{blocks} for definition) corresponding to a subset $U$ of $\mathcal{S}=\{1,\ldots,n\}$. Note that $\bigcup_{U\subseteq \mathcal{S}}\mathcal{B}(U)=\mathbb{Z}^n$. The following observation shows that if a module shifted by degree $\underline{-1}$ is quasi-straight, then its components are isomorphic in each block.
\begin{lemma}\label{iso on each block}
Let    $M(\underline{-1})$ be a quasi-straight $S$-module. Let $U$ be a subset of $\{1,\ldots,n\}$. Then, $M_{\underline{u}}\cong M_{\underline{v}}$ for $u,v\in \mathcal{B}(U)$.
\end{lemma}
\begin{proof}
Since $M(\underline{-1})$ is quasi-straight, $M(\underline{-1})_{\underline{u}}\cong M(\underline{-1})_{\underline{u}+e_i}$ whenever $\operatorname{supp}_{+}(\underline{u}+e_i)= \operatorname{supp}_{+}(\underline{u})$. This implies that $M_{\underline{u}}\cong M_{\underline{u}+e_i}$ if $\operatorname{supp}_{+}(\underline{u}+e_1+\ldots+e_n)= \operatorname{supp}_{+}(\underline{u}+e_1+\ldots+2e_i+\ldots+e_n)$.

    Let $U=\{r_1,r_2,\ldots,r_p\}\subseteq \mathcal{S}$. 
    Let $\underline{u}=(u_1,\ldots,u_n)$ and $\underline{v}=(v_1,\ldots,v_n)$ be two elements of $\mathcal{B}(U)$. Let $1\leq j\leq p$. By definition, we have $u_{r_j}\geq 0$. Since,
    \begin{align*}
        \operatorname{supp}_{+}(u_1+1,\ldots,1,\ldots,u_n+1)\ &=\operatorname{supp}_{+}(u_1+1,\ldots,2,\ldots,u_n+1)\\&= \operatorname{supp}_{+}(u_1+1,\ldots,3,\ldots,u_n+1)\\& \multicolumn{2}{c}{\vdots}  \\&= \operatorname{supp}_{+}(u_1+1,\ldots,u_{r_j}+1,\ldots,u_n+1)\\& =\operatorname{supp}_{+}(u_1+1,\ldots,u_{r_j}+2,\ldots,u_n+1)\\& \multicolumn{2}{c}{\vdots}
    \end{align*}
    we have 
    \begin{align*}
        M_{(u_1,\ldots,0,\ldots,u_n)}\ &\cong M_{(u_1,\ldots,1,\ldots,u_n)}\\&\cong  M_{(u_1,\ldots,2,\ldots,u_n)}\\& \multicolumn{2}{c}{\vdots}  \\&\cong M_{(u_1,\ldots,u_{r_j},\ldots,u_n)} \\&\cong M_{(u_1,\ldots,u_{r_j}+1,\ldots,u_n)}\\& \multicolumn{2}{c}{\vdots}
    \end{align*}
     Hence,
$M_{(u_1,\ldots,u_{r_j},\ldots,u_n)}\cong M_{(u_1,\ldots,u_{r_j}+s,\ldots,u_n)}$ for all $s\geq - u_{r_j}$. Let $q$ be such that $q\neq r_i$ for any $1\leq i\leq p$. Then a similar argument proves that $M_{(u_1,\ldots,u_q,\ldots,u_n)}\cong M_{(u_1,\ldots,u_{q}-s,\ldots,u_n)}$ for all $s\geq u_q+1$. We note that $v_{r_j}=u_{r_j}+a$ for some $a\geq - u_{r_j}$. Also if $q$ is such that $q\neq r_i$ for any $1\leq i\leq p$ then $v_q=u_q-b$ for some $b\geq u_q+1$. This implies that $M_{\underline{u}}\cong M_{\underline{v}}$.
\end{proof}
The following result examines the behavior of of straight modules in an exact sequence under certain assumptions. More precisely, it provides a criterion for straightness of a  $\mathbb{Z}^n$-graded submodule and $\mathbb{Z}^n$-graded quotients of a straight module.
\begin{proposition}
    \label{U and V are straight}
    Let $K$ be a field of any characteristic and $X$ be a $\mathbb{Z}^n$-graded Eulerian module over $\D(K[X_1,\ldots,X_n])$. Let $0\rightarrow U\rightarrow X\rightarrow V\rightarrow 0$ be an exact sequence of $\mathbb{Z}^n$-graded $\D(K[X_1,\ldots,X_n])$-modules and $X(\underline{-1})$ is straight, then $U(\underline{-1})$ and $V(\underline{-1})$ are both straight modules over $K[X_1,\ldots,X_n]$.
\end{proposition}
\begin{proof}
     Since $X$ is $\mathbb{Z}^n$-graded Eulerian, so is $V$. Consider the following commutative diagram

\begin{equation}\label{equation 1 for proving straight lemma}
\begin{tikzcd}
	0 & U(\underline{-1})_{\underline{u}} & X(\underline{-1})_{\underline{u}} & V(\underline{-1})_{\underline{u}} & 0  \\
	 0 & U(\underline{-1})_{\underline{u}+e_i} & X(\underline{-1})_{\underline{u}+e_i} & V(\underline{-1})_{\underline{u}+e_i}& 0 
	\arrow[ "{X_i}"',from=1-2, to=2-2]
    \arrow[ from=1-1, to=1-2]
    \arrow[ from=2-1, to=2-2]
    \arrow[from=1-3, to=1-4]
    \arrow[from=1-4, to=1-5]
    \arrow[from=2-3, to=2-4]
    \arrow[from=2-4, to=2-5]
    \arrow[ from=1-2, to=1-3]
	\arrow["{X_i}"', from=1-3, to=2-3]
    \arrow["{X_i}"', from=1-4, to=2-4]
	\arrow[ from=2-2, to=2-3]
\end{tikzcd}
\end{equation}
Since $X(\underline{-1})$ is a straight module, the vertical map $X(\underline{-1})_{\underline{u}}\xrightarrow{X_i}X(\underline{-1})_{\underline{u}+e_i}$ of (\ref{equation 1 for proving straight lemma}) is bijective and therefore, the third vertical map $V(\underline{-1})_{\underline{u}} \xrightarrow{X_i} V(\underline{-1})_{\underline{u}+e_i}$ is surjective for all $\underline{u}\in \mathbb{Z}^n$  with $\operatorname{supp}_{+}(\underline{u}+e_i)= \operatorname{supp}_{+}(\underline{u})$. Consequently, $V_{\underline{u}} \xrightarrow{X_i} V_{\underline{u}+e_i}$ is surjective for all $\underline{u}\in \mathbb{Z}^n$  with $\operatorname{supp}_{+}(\underline{u}+e_1+e_2+\ldots +e_n)= \operatorname{supp}_{+}(\underline{u}+e_1+e_2+\ldots +2e_i+\ldots+e_n)$.

For a fixed $(u_1,\ldots,u_{i-1},\hat{u},u_{i+1},\ldots,u_n)\in \mathbb{Z}^{n-1}$, let $V_i=\bigoplus_{u\in \mathbb{Z}}V_{(u_1,\ldots,u_{i-1},u,u_{i+1},\ldots,u_n)}$. Clearly, $V_i$ is $\mathbb{Z}$-graded Eulerian module over $\D(F[X_i])$. Consider the following commutative diagram

\begin{equation}\label{3rd equation of 2nd part}
\begin{tikzcd}
	0 & \Gamma_{(X_i)}(V_i)_{u} & (V_i)_{u} & \frac {(V_i)_{u}}{\Gamma_{(X_i)}(V_i)_{u}} & 0  \\
	0 & \Gamma_{(X_i)}(V_i)_{u+1} & (V_i)_{u+1} & \frac{(V_i)_{u+1}}{\Gamma_{(X_i)}(V_i)_{u+1}} & 0  
	\arrow[from=1-1, to=1-2]
	\arrow[ from=1-2, to=1-3]
    \arrow[ from=2-4, to=2-5]
    \arrow[ "{X_i}"', from=1-4, to=2-4]
	\arrow["{X_i}"', from=1-2, to=2-2]
	\arrow[ from=1-3, to=1-4]
    \arrow[ from=1-4, to=1-5]
	\arrow["{X_i}"', from=1-3, to=2-3]
	\arrow[ from=2-1, to=2-2]
	\arrow[ from=2-2, to=2-3]
	\arrow[ from=2-3, to=2-4]
\end{tikzcd}
\end{equation}
 Since $$\displaystyle\ker\left(\frac {(V_i)_{u}}{\Gamma_{(X_i)}(V_i)_{u}}\xlongrightarrow{X_i}\frac{(V_i)_{u+1}}{\Gamma_{(X_i)}(V_i)_{u+1}}\right)=0,$$ we get $$\ker\left(\Gamma_{(X_i)}(V_i)_{u}\xrightarrow{X_i}\Gamma_{(X_i)}(V_i)_{u+1}\right)=\ker\left((V_i)_{u}\xrightarrow{X_i}(V_i)_{u+1}\right).$$ By Theorem \ref{theorem of ma and zhang direct sum}, $\Gamma_{(X_i)}(V_i)$ is isomorphic to a direct sum (may be infinite) of $^*E(1)$ where $^*E$ is the graded injective of $F$ as $F[X_i]$-module. Hence, $\Gamma_{(X_i)}(V_i)_u\xrightarrow{X_i} \Gamma_{(X_i)}(V_i)_{u+1}$ is an isomorphism if $u\geq 0$ or if $u\leq -2$. Therefore, $\displaystyle(V_i)_{u}\xrightarrow{X_i} (V_i)_{u+1}$  i.e., $\displaystyle V_{(u_1,\ldots,u_{i-1},u,u_{i+1},\ldots,u_n)}\xrightarrow{X_i} V_{(u_1,\ldots,u_{i-1},u+1,u_{i+1},\ldots,u_n)}$ is an isomorphism if $u\geq 0$ \! and $\displaystyle\operatorname{supp}_{+}((u_1,\ldots,u_{i-1},u,u_{i+1},\ldots,u_n)+e_1+e_2+\ldots +e_n)= \operatorname{supp}_{+}((u_1,\ldots,u_{i-1},u,u_{i+1},\ldots,u_n)+e_1+e_2+\ldots+2e_i+\ldots +e_n)$. 

Similarly, we get that if  $u\leq -2$ and $\operatorname{supp}_{+}((u_1,\ldots,u_{i-1},u,u_{i+1},\ldots,u_n)+e_1+e_2+\ldots +e_n)= \operatorname{supp}_{+}((u_1,\ldots,u_{i-1},u,u_{i+1},\ldots,u_n)+e_1+e_2+\ldots+2e_i+\ldots +e_n)$, then $V_{(u_1,\ldots,u_{i-1},u,u_{i+1},\ldots,u_n)}\xrightarrow{X_i}V_{(u_1,\ldots,u_{i-1},u+1,u_{i+1},\ldots,u_n)}$ is an isomorphsm. 

Also if $u=-1$, then we see that $\operatorname{supp}_{+}((u_1,\ldots,u_{i-1},u,u_{i+1},\ldots,u_n)+e_1+e_2+\ldots +e_n)\neq  \operatorname{supp}_{+}((u_1,\ldots,u_{i-1},u,u_{i+1},\ldots,u_n)+e_1+e_2+\ldots+2e_i+\ldots +e_n)$. This implies that $V_{\underline{u}} \xrightarrow{X_i} V_{\underline{u}+e_i}$ is bijective for all $\underline{u}\in \mathbb{Z}^n$  with $\operatorname{supp}_{+}(\underline{u}+e_1+e_2+\ldots +e_n)= \operatorname{supp}_{+}(\underline{u}+e_1+e_2+\ldots +2e_i+\ldots+e_n)$ and therefore $V(\underline{-1})$  is a straight module.

Now from the diagram (\ref{equation 1 for proving straight lemma}) using Snake lemma, we see that  $U(\underline{-1})_{\underline{u}} \xrightarrow{X_i} U(\underline{-1})_{\underline{u}+e_i}$ is bijective for all $\underline{u}\in \mathbb{Z}^n$  with $\operatorname{supp}_{+}(\underline{u}+e_i)= \operatorname{supp}_{+}(\underline{u})$ since the other two vertical maps are bijective for all $\underline{u}\in \mathbb{Z}^n$  with $\operatorname{supp}_{+}(\underline{u}+e_i)= \operatorname{supp}_{+}(\underline{u})$. This implies that $U(\underline{-1})$  is a straight module.
\end{proof}
The above proposition is useful to establish two corollaries below that are important for the subsequent sections and are used to prove our main results. Before that we need to note the following crucial remark.
\begin{remark}\label{remark of multiplication map by pi is linear}
\normalfont
    Let the hypothesis be as in Setup \ref{main setup of our paper}. We have that $R=A[X_1,\ldots,X_n]$ where $(A,\pi A,F)$ is a complete DVR. Let $F= A/\pi A$ and $S=F[X_1,\ldots,X_n]$.
    Then the exact sequence $0\rightarrow R\xrightarrow{\pi}R\rightarrow S\rightarrow 0$ induces the following long exact sequence of local cohomology modules
$$\cdots\rightarrow H^{i-1}_{IS}(S)\rightarrow H^i_I(R)\xrightarrow{\pi} H^i_I(R)\rightarrow H^i_{IS}(S)\rightarrow\cdots.$$
Since, $H^i_I(R)$ is a $\D_A(R)$-module  and $\pi\in A$, the multiplication map $H^i_I(R)\xrightarrow{\pi} H^i_I(R)$ is $\D_A(R)$-linear. Therefore, $\ker\left(H^i_I(R)\xrightarrow{\pi} H^i_I(R)\right)$ and $\operatorname{coker}\left(H^i_I(R)\xrightarrow{\pi} H^i_I(R)\right)$ are  $\D_A(R)$-modules and hence $\D_F(S)$-modules. Consequently, the natural  maps
$H^{i-1}_{IS}(S)\rightarrow \ker\left(H^i_I(R)\xrightarrow{\pi} H^i_I(R)\right) $ and
$\operatorname{coker}\left(H^i_I(R)\xrightarrow{\pi} H^i_I(R)\right)\rightarrow H^i_{IS}(S)$ are morphisms in the category of $\D_A(R)$-modules as well as in the category of $\D_F(S)$-modules.
\end{remark}
We now state and prove the two corollaries. The first one is as follows.
\begin{corollary}\label{kernel and cokernel straight}
    Let the hypothesis be as in Setup \ref{main setup of our paper} and let $S=F[X_1,\ldots,X_n]$. Then, $M/\pi M(\underline{-1})$ and $\ker(M\xrightarrow{\pi}M)(\underline{-1})$ are straight as  $S$-modules.
\end{corollary}
\begin{proof}
 Let $F=A/\pi A$ and $S=F[X_1,\ldots,X_n]$. 
By Remark \ref{remark of multiplication map by pi is linear}, we have the following exact sequence  in the category of $\D_F(S)$-modules
$$H^{i-1}_{IS}(S)\rightarrow \ker(M\xrightarrow{\pi}M)\rightarrow 0. $$
Since $H^{i-1}_{IS}(S)(\underline{-1})$ is straight by Theorem \ref{theorem of terai and mustata}  and $H^{i-1}_{IS}(S)$ is $\mathbb{Z}^n$-graded Eulerian by Remark \ref{remark of multigraded eulerian}, we get that $\ker(M\xrightarrow{\pi}M)(\underline{-1})$ is straight by Proposition \ref{U and V are straight}.

Again  by Remark \ref{remark of multiplication map by pi is linear}, we have the following exact sequence in the category of $\D_F(S)$-modules $$0\rightarrow \frac{M}{\pi M}\rightarrow H^i_{IS}(S).$$
 Since $H^{i}_{IS}(S)(\underline{-1})$ is straight by Theorem \ref{theorem of terai and mustata}  and $H^{i}_{IS}(S)$ is $\mathbb{Z}^n$-graded Eulerian by Remark \ref{remark of multigraded eulerian}, we get that $M/\pi M(\underline{-1})$ is straight by Proposition \ref{U and V are straight}.
\end{proof}
\begin{corollary}\label{kernel of gamma and bar of M straight}
    Let the hypothesis be as in Setup \ref{main setup of our paper} and let $S=F[X_1,\ldots,X_n]$. Then, the modules $\ker\left(\Gamma_{\pi}(M)\xrightarrow{\pi}\Gamma_{\pi}(M)\right)\!(\underline{-1})$, $\Gamma_{\pi}(M)/\pi\Gamma_{\pi}(M)(\underline{-1})$ and $\overline{M}/\pi\overline{M}(\underline{-1})$ and  are straight as $S$-modules.
\end{corollary}
\begin{proof}
  Consider the following commutative diagram
\begin{equation}\label{equation of gamma M}
\begin{tikzcd}
	0 & \Gamma_\pi(M) & M & \frac{M}{\Gamma_\pi(M)} & 0  \\
	0 & \Gamma_\pi(M) & M & \frac{M}{\Gamma_\pi(M)} & 0
	\arrow[from=1-1, to=1-2]
	\arrow[ from=1-2, to=1-3]
    \arrow[ from=2-4, to=2-5]
    \arrow[ "{\pi}"', from=1-4, to=2-4]
	\arrow["{\pi}"', from=1-2, to=2-2]
	\arrow[ from=1-3, to=1-4]
    \arrow[ from=1-4, to=1-5]
	\arrow["{\pi}"', from=1-3, to=2-3]
	\arrow[ from=2-1, to=2-2]
	\arrow[ from=2-2, to=2-3]
	\arrow[ from=2-3, to=2-4]
\end{tikzcd}
\end{equation}  
Since $\ker\left(M/\Gamma_\pi(M)\xrightarrow{\pi}M/\Gamma_\pi(M)\right)=0$, we get $\ker\left(M\xrightarrow{\pi}M\right)=\ker\left(\Gamma_\pi(M)\xrightarrow{\pi}\Gamma_\pi(M)\right)$. Therefore, $\ker\left(\Gamma_{\pi}(M)\xrightarrow{\pi}\Gamma_{\pi}(M)\right)(\underline{-1})$ is straight by Lemma \ref{kernel and cokernel straight}.

Now the diagram \ref{equation of gamma M} induces the following exact sequence of $\D_F(S)$-modules 
$$0\rightarrow \frac{\Gamma_{\pi}(M)}{\pi\Gamma_{\pi}(M)}\rightarrow \frac{M}{\pi M}\rightarrow \frac{\overline{M}}{\pi \overline{M}}\rightarrow 0.$$
Since $M/\pi M$ is a submodule of $H^i_{IS}(S)$ (see , Remark \ref{remark of multiplication map by pi is linear}) and the later module is $\mathbb{Z}^n$-graded Eulerian, so is  $M/\pi M$. Also from Corollary \ref{kernel and cokernel straight}, we get that $M/\pi M(\underline{-1})$ is straight.  The result now follows from Proposition \ref{U and V are straight}.
\end{proof}
\section{Proof of Theorem A}
In this section, we give a proof of \hyperref[Theorem A]{Theorem A}. We recall the theorem again for the convenience of the reader.
\begin{theorem}\label{proof of theorem a}
    Let the hypothesis be as in Setup \ref{main setup of our paper}. Then, the sets  $\{a(\underline{u})\mid \underline{u}\in \mathbb{Z}^n\}$ and $\{b(\underline{u})\mid \underline{u}\in \mathbb{Z}^n\}$ are constant on $\mathcal{B}(U)$ for each subset $U$ of $\{1, \ldots, n\}$.
\end{theorem}
\begin{proof}
Let  $U$ be a subset of $\{1, \ldots, n\}$. Consider the exact sequence of $A$-modules
\begin{equation*} 
    0\rightarrow \Gamma_\pi(M_{\underline{u}})\rightarrow M_{\underline{u}}\rightarrow \frac{M_{\underline{u}}}{\Gamma_\pi(M_{\underline{u}})}\rightarrow 0 .
\end{equation*}

Tensoring the above sequence with $K$, we get the following exact sequence  
$$0\rightarrow \Gamma_\pi(M_{\underline{u}})\otimes_A K\rightarrow M_{\underline{u}}\otimes_A K\rightarrow \frac{M_{\underline{u}}}{\Gamma_\pi(M_{\underline{u}})}\otimes_A K\rightarrow 0 .$$
Since $\Gamma_\pi(M_{\underline{u}})\otimes_A K=0$, we get that $M_{\underline{u}}\otimes_A K\cong M_{\underline{u}}/\Gamma_\pi(M_{\underline{u}})\otimes_A K\cong  K^{a(\underline{u})+b(\underline{u})}.$ Note that $M\otimes_A K=H^i_I(K[X_1,\ldots,X_n])$. By Theorem \ref{theorem of terai and mustata}, $H^i_I(K[X_1,\ldots,X_n])(\underline{-1})$ is straight. Consequently, by Lemma \ref{iso on each block},   the set $\{a(\underline{u})+b(\underline{u})\mid \underline{u}\in \mathbb{Z}^n\}$ is constant  on $\mathcal{B}(U)$. 

Since, 
$$\overline{M_{\underline{u}}}= \frac{M_{\underline{u}}}{\Gamma_{\pi}(M_{\underline{u}})}=A^{a(\underline{u})}\oplus K^{b(\underline{u})},$$
we get $$\frac{\overline{M_{\underline{u}}}}{\pi \overline{M_{\underline{u}}}}=(A/\pi)^{a(\underline{u})}.$$
From Lemma \ref{kernel of gamma and bar of M straight}, $\overline{M}/\pi\overline{M}(\underline{-1})$ is straight. Therefore by Lemma \ref{iso on each block}, the set $\{a(\underline{u})\mid \underline{u}\in \mathbb{Z}^n\}$ is constant  on $\mathcal{B}(U)$. This implies that $\{b(\underline{u})\mid \underline{u}\in \mathbb{Z}^n\}$ is constant  on $\mathcal{B}(U)$ because $\{a(\underline{u})+b(\underline{u})\mid \underline{u}\in \mathbb{Z}^n\}$ is constant  on $\mathcal{B}(U)$.

\end{proof}
\section{Proof of  Theorem B}
In this section, we give a proof of \hyperref[Theorem B]{Theorem B}, which studies the behaviour of the sets  $\{l(\underline{u})\mid \underline{u}\in \mathbb{Z}^n\}$ and $\{t(\underline{u})\mid \underline{u}\in \mathbb{Z}^n\}$ on $\mathcal{B}(U)$ for each subset $U$ of $\{1, \ldots, n\}$. 
\begin{theorem}\label{proof of theorem b}
    Let the hypothesis be as in Setup \ref{main setup of our paper}. Then, the sets  $\{l(\underline{u})\mid \underline{u}\in \mathbb{Z}^n\}$ and $\{t(\underline{u})\mid \underline{u}\in \mathbb{Z}^n\}$ are constant on $\mathcal{B}(U)$ for each subset $U$ of $\{1, \ldots, n\}$.
\end{theorem}
\begin{proof}
Let $F= A/\pi A$ and $S=F[X_1,\ldots,X_n]$. 
Consider the following commutative diagram
\begin{equation}\label{first equation for part b}
\begin{tikzcd}
	0 & \Gamma_\pi(M_{\underline{u}}) & M_{\underline{u}} & \frac{M_{\underline{u}}}{\Gamma_\pi(M_{\underline{u}})} & 0  \\
	0 & \Gamma_\pi(M_{\underline{u}}) & M_{\underline{u}} & \frac{M_{\underline{u}}}{\Gamma_\pi(M_{\underline{u}})} & 0 
	\arrow[from=1-1, to=1-2]
	\arrow[ from=1-2, to=1-3]
    \arrow[ from=2-4, to=2-5]
    \arrow[ "{\pi}"', from=1-4, to=2-4]
	\arrow["{\pi}"', from=1-2, to=2-2]
	\arrow[ from=1-3, to=1-4]
    \arrow[ from=1-4, to=1-5]
	\arrow["{\pi}"', from=1-3, to=2-3]
	\arrow[ from=2-1, to=2-2]
	\arrow[ from=2-2, to=2-3]
	\arrow[ from=2-3, to=2-4]
\end{tikzcd}
\end{equation}
Note that
$$\Gamma_{\pi}(M_{\underline{u}})=E_A\left(\frac{A}{\pi A}\right)^{l({\underline{u}})} \oplus\left(\bigoplus_{j=1}^{t({\underline{u}})}\frac{A}{\pi^{\beta_j({\underline{u}})}A}\right).$$
Since $(A,\pi A)$ is a DVR, the kernel of the multiplication map $E(A/\pi A)\xrightarrow{\pi} E(A/\pi A)$ is $A/\pi A=F$. Also for any $r\geq 0$, the kernel of the map $A/\pi^r A\xrightarrow{\pi} A/\pi^r A$ is $\pi^{r-1}A/\pi^rA\cong A/\pi A=F$. Hence,  $\ker\left(\Gamma_{\pi}(M_{\underline{u}})\xrightarrow{\pi} \Gamma_{\pi}(M_{\underline{u}})\right)$ is $F^{l(\underline{u})+t(\underline{u})}$. It is easy to see that $\ker\left(M_{\underline{u}}/\Gamma_\pi(M_{\underline{u}})\xrightarrow{\pi}M_{\underline{u}}/\Gamma_\pi(M_{\underline{u}})\right)=0$ since $M_{\underline{u}}/\Gamma_\pi(M_{\underline{u}})= A^{a(\underline{u})}\oplus K^{b(\underline{u})}.$ Therefore, Snake lemma applied to (\ref{first equation for part b}) gives $\ker(M_{\underline{u}}\xrightarrow{\pi}M_{\underline{u}})=F^{l(\underline{u})+t(\underline{u})}$.
Since $\ker(M\xrightarrow{\pi}M)(\underline{-1})$ is straight by Lemma \ref{kernel and cokernel straight}, the set  $\{l(\underline{u})+t(\underline{u})\mid \underline{u}\in \mathbb{Z}^n\}$ is constant on $\mathcal{B}(U)$ by Lemma \ref{iso on each block}.

Now $$\Gamma_{\pi}(M_{\underline{u}})=E_A\left(\frac{A}{\pi A}\right)^{l({\underline{u}})} \oplus\left(\bigoplus_{j=1}^{t({\underline{u}})}\frac{A}{\pi^{\beta_j({\underline{u}})}A}\right)$$ implies 
$$\pi\Gamma_{\pi}(M_{\underline{u}})=E_A\left(\frac{A}{\pi A}\right)^{l({\underline{u}})} \oplus\left(\bigoplus_{j=1}^{t({\underline{u}})}\frac{\pi A+\pi^{\beta_j({\underline{u}})}A}{\pi^{\beta_j({\underline{u}})}A}\right)$$
since $\displaystyle E_A\left(A/\pi A\right)\xrightarrow{\pi}E_A\left(A/\pi A\right)$ is surjective.
Therefore, 
$$\frac{\Gamma_{\pi}(M_{\underline{u}})}{\pi \Gamma_{\pi}(M_{\underline{u}})}=\bigoplus_{j=1}^{t({\underline{u}})}\left(\frac{A}{\pi A}\right)=\left(\frac{A}{\pi A}\right)^{t({\underline{u}})}.$$
 Since $\displaystyle\Gamma_{\pi}(M)/\pi\Gamma_{\pi}(M)(\underline{-1})$ is straight by Lemma \ref{kernel of gamma and bar of M straight}, Lemma \ref{iso on each block} implies that $\{t(\underline{u})\mid \underline{u}\in \mathbb{Z}^n\}$ is constant on $\mathcal{B}(U)$. Again since $\{l(\underline{u})+t(\underline{u})\mid \underline{u}\in \mathbb{Z}^n\}$ is constant on $\mathcal{B}(U)$, we get that $\{l(\underline{u})\mid \underline{u}\in \mathbb{Z}^n\}$ is constant on $\mathcal{B}(U)$.
\end{proof}
\section{Proof of Theorem C and Theorem D}
Before proceeding to the proof \hyperref[Theorem D]{Theorem D}, we first establish some preparatory results. In the sequel, we begin with the following lemma.
\begin{lemma}\label{Eulerian module is 0}
Let $M$ be a $\mathbb{Z}$-graded Eulerian graded $\D(F[X])$-module. Let $i\geq 1$ such that $M_j=0$ for $|j|\leq i$. Then, $M=0$.
\end{lemma}
\begin{proof}
If $\Gamma_{(X)}(M)\neq 0$, then  by Theorem \ref{theorem of ma and zhang direct sum}, $\Gamma_{(X)}(M)$ is a direct sum (may be infinite) of $^*E(1)$ where $^*E$ is the graded injective of $F$ as $F[X]$-module.  Therefore, $\Gamma_{(X)}(M)_u\neq 0$ and hence $M_u\neq 0$ for $u\leq -1$ which is a contradiction. This implies that $\Gamma_{(X)}(M)= 0$ and hence $X$ is $M$-regular. Let $m\in M_u$ where $u\leq 0$. Choose a sufficiently large value of $r$ such that $X^rm\in M_0=\{0\}$. Consequently, $m=0$. This implies $M$ is zero in the negative axis.

We claim that $M/XM$ is concentrated in degree zero. The reason for this is the following:

Let $m\in M$ be homogeneous of $\deg(m)>  0$. Then, $\deg(m)=p^st$ such that $p$ does not divide $t$. Since $M$ is  Eulerian  $\D(F[X])$-module, we have $$E_{p^s}m=\binom{p^st}{p^s}m$$ where $\displaystyle E_{p^s}=X^{p^s}\frac{1}{p^s!}\frac{\partial^{p^s}}{\partial X^{p^s}}$. Since $p$ does not divide $\displaystyle\binom{p^st}{p^s}$, we get that $m\in XM$ and the claim follows.

Consider the exact sequence $$0\rightarrow M_{n-1}\xrightarrow{X} M_n\rightarrow \overline{M}_n\rightarrow0.$$
Since $\overline{M}$ is concentrated in degree zero, $M_{n-1}\cong M_n$ for $n\neq 0$. Now for $u\geq 1$, $M_u\cong M_1=\{0\}$. This proves our Lemma.
\end{proof}
As an immediate corollary of the above lemma, we obtain the following.
\begin{corollary}\label{eulerian module 0 in Z^n}
    Let $M$ be a $\mathbb{Z}^n$-graded Eulerian $\D(F[X_1,
    \ldots,X_n])$-module such that $M_{\underline{u}}=0$ for all $\underline{u}$ in $[-1,1]^n\cap \mathbb{Z}^n$. Then, $M=0$.
\end{corollary}
\begin{proof}
We prove the result by induction on $n$. If $n=1$, the result follows from Lemma \ref{Eulerian module is 0}. Let $n=2$. Consider the $\mathbb{Z} $-graded Eulerian $\D(F[X_2])$-module $N=\bigoplus_{i\in \mathbb{Z}} M_{(1,i)}$. Since $N_1=N_0=N_{-1}=0$, by Lemma \ref{Eulerian module is 0}, $N=0$. Similarly, the modules  $\bigoplus_{i\in \mathbb{Z}} M_{(-1,i)}$ and $\bigoplus_{i\in \mathbb{Z}} M_{(0,i)}$ are zero. Again replacing  $X_2$ by $X_1$, we see that the modules $\bigoplus_{i\in \mathbb{Z}} M_{(i,-1)}$, $\bigoplus_{i\in \mathbb{Z}} M_{(i,0)}$ and $\bigoplus_{i\in \mathbb{Z}} M_{(i,1)} $ are zero. Therefore, we see that $M=0$ on the lines $X_1=\pm 1,0$ and $X_2=\pm 1,0$. Now we take any arbitrary point say $(a,b)\in \mathbb{Z}^2$. Consider $N'=M_{(a,i)}$. Since $N'$ is Eulerian as $D(F[X_2])$-module and $N'_1=N'_0=N'_{-1}=0$, the Lemma \ref{Eulerian module is 0} implies $N'=0$. Therefore, $M_{(a,b)}=0$ and hence $M=0$.

Assume the result for $n-1$ and we prove it for $n$. The proof is similar to $n=2$ case. For $a\in \{-1,0,1\}$, consider $N^a=\displaystyle \bigoplus_{(a_1,\dots,a_{n-1}) \in \mathbb{Z}^{\,n-1}} M_{(a_1,\ldots,a_{n-1},a)}$. Since $N^a$ is zero on $[-1,1]^{n-1}\cap \mathbb{Z}^{n-1}$ and is $\mathbb{Z}^{n-1}$-graded Eulerian over $F[X_1,\ldots,X_{n-1}]$, the induction hypothesis implies $N^a=0$ on $\mathbb{Z}^{n-1}$. Therefore, we have shown that $M$ is zero on the hyperplanes $X_n=0,\pm 1$. Now take any arbitrary point say $(r_1,\ldots,r_n)\in \mathbb{Z}^n$. Let $T=\displaystyle \bigoplus_{(a_2,\dots,a_{n}) \in \mathbb{Z}^{\,n-1}} M_{(r_1,a_2,\ldots,a_{n})}$. Clearly, $T$ is zero on $[-1,1]^{n-1}\cap \mathbb{Z}^{n-1}$ and $\mathbb{Z}^{n-1}$-graded Eulerian $\D(F[X_2,\ldots,X_n])$-module. By induction hypothesis $T=0$. This completes the proof of the corollary.
\end{proof}
 Since $[-1,1]^n\cap \mathbb{Z}^n$ is finite, we can define a finite number as $m=\max\{\beta_j({\underline{u}})\mid 1\leq j\leq t({\underline{u}})\  \text{where} \ \underline{u}\in [-1,1]^n\cap\  \mathbb{Z}^n\}$. We prove that the set $\{\beta_j({\underline{u}})\mid 1\leq j\leq t({\underline{u}})\  \text{where} \ \underline{u}\in  \mathbb{Z}^n\}$ is bounded above by $m$. 
 \begin{theorem}\label{proof of proposition c}
     Let the hypothesis be as in Setup \ref{main setup of our paper} and $m$ be as in the last paragraph. Then, the set $\{\beta_j({\underline{u}})\mid 1\leq j\leq t({\underline{u}})\  \text{where} \ \underline{u}\in  \mathbb{Z}^n\}$ is finite. In fact,  $\beta_j({\underline{u}})\leq m$ for all $1\leq j\leq t({\underline{u}})$ where  $\underline{u}\in  \mathbb{Z}^n$.
 \end{theorem}

 \begin{proof}
We denote $H^i_I(R)$ by $M$, $\Gamma_{\pi}(H^i_I(R))$ by $L$ and $M/\Gamma_\pi(M)$ by $\overline{M}$. Let $S=F[X_1,\ldots,X_n]$ where $F=A/\pi A$. Then,
 $$L_{\underline{u}}=E_A\left(\frac{A}{\pi A}\right)^{l({\underline{u}})} \oplus\left(\bigoplus_{j=1}^{t({\underline{u}})}\frac{A}{\pi^{\beta_j({\underline{u}})}A}\right).$$

 By Remark \ref{remark of multiplication map by pi is linear}, we have the following exact sequence in the category of $\D_F(S)$-modules $$0\rightarrow \frac{M}{\pi M}\rightarrow H^i_{IS}(S).$$
 Since $M/\pi M$ is a submodule of $H^i_{IS}(S)$ and the later module is $\mathbb{Z}^n$-graded Eulerian, so is  $M/\pi M$.
  
The  commutative diagram
\begin{equation*}
\begin{tikzcd}
	0 & L & M & \overline{M} & 0  \\
	0 & L & M & \overline{M} & 0
	\arrow[from=1-1, to=1-2]
	\arrow[ from=1-2, to=1-3]
    \arrow[ from=2-4, to=2-5]
    \arrow[ "{\pi}"', from=1-4, to=2-4]
	\arrow["{\pi}"', from=1-2, to=2-2]
	\arrow[ from=1-3, to=1-4]
    \arrow[ from=1-4, to=1-5]
	\arrow["{\pi}"', from=1-3, to=2-3]
	\arrow[ from=2-1, to=2-2]
	\arrow[ from=2-2, to=2-3]
	\arrow[ from=2-3, to=2-4]
\end{tikzcd}
\end{equation*}
induces an exact sequence of $\D_F(S)$-modules as follows
$$0\rightarrow \frac{L}{\pi L}\rightarrow \frac{M}{\pi M}\rightarrow \frac{\overline{M}}{\pi \overline{M}}\rightarrow 0.$$
This implies $ L/\pi L$ is $\mathbb{Z}^n$-graded Eulerian as $\D_F(S)$-module since it is a submodule of $M/\pi M$. 

Note that $\pi$ is in the center of $\D_A(R)$. Let $T^g$ be the kernel of the surjective $\D_A(R)$-linear map $L\rightarrow \pi^{g+1}L$ for $g\geq 0$. The commutative diagram
\begin{equation*}
\begin{tikzcd}
	0 & T^g & L & \pi^g L & 0  \\
	0 & T^g & L & \pi^{g+1}L & 0
	\arrow[from=1-1, to=1-2]
	\arrow[ from=1-2, to=1-3]
    \arrow[ from=2-4, to=2-5]
    \arrow[ "{\pi}"', from=1-4, to=2-4]
	\arrow["{\pi}"', from=1-2, to=2-2]
	\arrow[ from=1-3, to=1-4]
    \arrow[ from=1-4, to=1-5]
	\arrow["{\pi}"', from=1-3, to=2-3]
	\arrow[ from=2-1, to=2-2]
	\arrow[ from=2-2, to=2-3]
	\arrow[ from=2-3, to=2-4]
\end{tikzcd}
\end{equation*} gives the following exact sequence in the category of of $\D_F(S)$-modules
$$ \frac{T^g}{\pi T^g}\rightarrow \frac{L}{\pi L}\rightarrow \frac{\pi^g L}{\pi^{g+1}L}\rightarrow 0.$$
Since $L/\pi L$ is $\mathbb{Z}^n$-graded Eulerian, so is $\pi^g L/\pi^{g+1} L$ as it is a quotient of $L/\pi L$. In particular, when $g=m$, we get that $\pi^m L/\pi^{m+1} L$ is Eulerian. Again note that $\left(\pi^m L/\pi^{m+1} L\right)_{\underline{u}}=0$ for all $\underline{u}$ in $[-1,1]^n\cap \mathbb{Z}^n$. Hence, $\pi^m L/\pi^{m+1} L=0$ by Corollary \ref{eulerian module 0 in Z^n}. Therefore, $ \pi^{m+1} \left(A/\pi^{\beta_j({\underline{u}})}A\right)=\pi^{m} \left(A/\pi^{\beta_j({\underline{u}})}A\right)$ and hence by Nakayama Lemma, $\pi^{m} \left(A/\pi^{\beta_j({\underline{u}})}A\right)=0$ for all $1\leq j\leq t({\underline{u}})$ where  $\underline{u}\in  \mathbb{Z}^n$. Consequently, $\beta_j({\underline{u}})\leq m$ for all $1\leq j\leq t({\underline{u}})$ where  $\underline{u}\in  \mathbb{Z}^n$.
\end{proof}
Let us recall the setup \ref{second setup} below.
\begin{setup}\label{second setup'}
Let $s=\max\{\beta_j({\underline{u}})\mid 1\leq j\leq t({\underline{u}})\  \text{where} \ \underline{u}\in  \mathbb{Z}^n\}$. By Theorem \ref{proof of proposition c}, $s$ is a finite number and $s\leq m$. Let the hypothesis be as in Setup \ref{main setup of our paper}. Then,
$$
\Gamma_{\pi}(M_{\underline{u}})=E_A\left(\frac{A}{\pi A}\right)^{l({\underline{u}})} \oplus\left(\bigoplus_{j=1}^{s}\left(\frac{A}{\pi^jA}\right)^{\alpha_j(\underline{u})}\right).
$$
\end{setup}
 We should again remark that in the above expression of $\Gamma_{\pi}(M_{\underline{u}})$, it may happen that $\alpha_j(\underline{u})$ is zero for some $1\leq j\leq s$, precisely when the summand $A/\pi^jA$ does not appear in the expression. In the next result, we study the nature of the numbers $\alpha_j(\underline{u})$ appearing in the above expression. More precisely, we prove the following result.
\begin{theorem}\label{proof of theorem D}
    Let $1\leq j\leq s$. Let the hypothesis be as in Setup \ref{second setup'}. Then, the set $\{\alpha_j(\underline{u})\mid \underline{u}\in \mathbb{Z}^n\}$ is constant on $\mathcal{B}(U)$ for each subset $U$ of $\{1, \ldots, n\}$.
\end{theorem}
\begin{proof}
We denote $H^i_I(R)$ by $M$, $\Gamma_{\pi}(H^i_I(R))$ by $L$ and $M/\Gamma_\pi(M)$ by $\overline{M}$. Let $S=F[X_1,\ldots,X_n]$ where $F=A/\pi A$.  From the proof of Theorem \ref{proof of proposition c}, we have the following exact sequence in the category of of $\D_F(S)$-modules
$$ \frac{T^g}{\pi T^g}\rightarrow \frac{L}{\pi L}\rightarrow \frac{\pi^g L}{\pi^{g+1}L}\rightarrow 0$$ for all $g\geq 0$ and $ L/\pi L$ is $\mathbb{Z}^n$-graded Eulerian as $\D_F(S)$-module. By Corollary \ref{kernel of gamma and bar of M straight}, $L/\pi L(\underline{-1})$ is straight. Therefore, by Proposition \ref{U and V are straight}, $\pi^g L/\pi^{g+1}L(\underline{-1})$ is straight for all $g\geq 0$.

We have
    $$L_{\underline{u}}=E_A\left(\frac{A}{\pi A}\right)^{l({\underline{u}})} \oplus\left(\bigoplus_{j=1}^{s}\left(\frac{A}{\pi^jA}\right)^{\alpha_j(\underline{u})}\right).$$
    Note that $\displaystyle E_A(A/\pi A)\xrightarrow{\pi }E_A(A/\pi A)$ is surjective. This implies 
    $$\pi^sL_{\underline{u}}=E_A\left(\frac{A}{\pi A}\right)^{l({\underline{u}})} \oplus\left(\bigoplus_{j=1}^{s}\left(\pi^s\frac{A}{\pi^jA}\right)^{\alpha_j(\underline{u})}\right)=E_A\left(\frac{A}{\pi A}\right)^{l({\underline{u}})}$$
     and 
     \begin{align*}
         \pi^{s-1}L_{\underline{u}}\ &=E_A\left(\frac{A}{\pi A}\right)^{l({\underline{u}})} \oplus\left(\bigoplus_{j=1}^{s}\left(\pi^{s-1}\frac{A}{\pi^jA}\right)^{\alpha_j(\underline{u})}\right)\\&=E_A\left(\frac{A}{\pi A}\right)^{l({\underline{u}})} \oplus\left(\bigoplus_{j=1}^{s}\left(\frac{\pi^{s-1}A+\pi^jA}{\pi^jA}\right)^{\alpha_j(\underline{u})}\right)\\&= E_A\left(\frac{A}{\pi A}\right)^{l({\underline{u}})} \oplus\left(\frac{\pi^{s-1}A}{\pi^{s}A}\right)^{\alpha_s(\underline{u})}.
     \end{align*}
      Therefore,
     $$\frac{\pi^{s-1}L_{\underline{u}}}{\pi^{s}L_{\underline{u}}}\cong\left(\frac{\pi^{s-1}A}{\pi^{s}A}\right)^{\alpha_s(\underline{u})}\cong\left(\frac{A}{\pi A}\right)^{\alpha_s(\underline{u})}.$$ 
    
Since $\displaystyle\pi^{s-1}L/\pi^s L(\underline{-1})$ is straight,  the set $\{\alpha_s(\underline{u})\mid \underline{u}\in \mathbb{Z}^n\}$ is constant on $\mathcal{B}(U)$ for each subset $U$ of $\{1, \ldots, n\}$ by Lemma \ref{iso on each block}.

      Now assume that the sets $\{\alpha_{s-r}(\underline{u})\mid \underline{u}\in \mathbb{Z}^n\}$, $\{\alpha_{s-r+1}(\underline{u})\mid \underline{u}\in \mathbb{Z}^n\}$,\ldots,  $\{\alpha_{s}(\underline{u})\mid \underline{u}\in \mathbb{Z}^n\}$ are constant on $\mathcal{B}(U)$ for some $r\geq 0$. We prove that the set $\{\alpha_{s-r-1}(\underline{u})\mid \underline{u}\in \mathbb{Z}^n\}$ is constant on $\mathcal{B}(U)$.
      Now 
      \begin{align*}
         \pi^{s-r-2}L_{\underline{u}}\ &=E_A\left(\frac{A}{\pi A}\right)^{l({\underline{u}})} \oplus\left(\bigoplus_{j=1}^{s}\left(\pi^{s-r-2}\frac{A}{\pi^jA}\right)^{\alpha_j(\underline{u})}\right)\\&=E_A\left(\frac{A}{\pi A}\right)^{l({\underline{u}})} \oplus\!\left(\bigoplus_{j=1}^{s}\left(\frac{\pi^{s-r-2}A+\pi^j A}{\pi^j A}\right)^{\alpha_j(\underline{u})}\right)\\&= E_A\left(\frac{A}{\pi A}\right)^{l({\underline{u}})}\oplus\left(\frac{\pi^{s-r-2}A}{\pi^{s-r-1}A}\right)^{\alpha_{s-r-1}(\underline{u})}\oplus\ldots\oplus\left(\frac{\pi^{s-r-2}A}{\pi^{s-1}A}\right)^{\alpha_{s-1}(\underline{u})}\oplus\left(\frac{\pi^{s-r-2}A}{\pi^{s}A}\right)^{\alpha_{s}(\underline{u})}
     \end{align*}
     and 
      \begin{align*}
         \pi^{s-r-1}L_{\underline{u}}\ &=E_A\left(\frac{A}{\pi A}\right)^{l({\underline{u}})} \oplus\left(\bigoplus_{j=1}^{s}\left(\pi^{s-r-1}\frac{A}{\pi^jA}\right)^{\alpha_j(\underline{u})}\right)\\&=E_A\left(\frac{A}{\pi A}\right)^{l({\underline{u}})} \oplus\left(\bigoplus_{j=1}^{s}\left(\frac{\pi^{s-r-1}A+\pi^jA}{\pi^jA}\right)^{\alpha_j(\underline{u})}\right)\\&= E_A\left(\frac{A}{\pi A}\right)^{l({\underline{u}})} \oplus\left(\frac{\pi^{s-r-1}A}{\pi^{s-r}A}\right)^{\alpha_{s-r}(\underline{u})}\oplus\ldots\oplus\left(\frac{\pi^{s-r-1}A}{\pi^{s-1}A}\right)^{\alpha_{s-1}(\underline{u})}\oplus\left(\frac{\pi^{s-r-1}A}{\pi^{s}A}\right)^{\alpha_{s}(\underline{u})}.
     \end{align*}
      Therefore,
      \begin{align*}
         \frac{\pi^{s-r-2}L_{\underline{u}}}{\pi^{s-r-1}L_{\underline{u}}}\ & \cong\left(\frac{\pi^{s-r-2}A}{\pi^{s-r-1}A}\right)^{\alpha_{s-r-1}(\underline{u})}\oplus \left(\frac{\pi^{s-r-2}A}{\pi^{s-r-1}A}\right)^{\alpha_{s-r}(\underline{u})}\oplus\ldots\oplus \left(\frac{\pi^{s-r-2}A}{\pi^{s-r-1}A}\right)^{\alpha_{s}(\underline{u})} \\&\cong\left(\frac{A}{\pi A}\right)^{\alpha_{s-r-1}(\underline{u})+\alpha_{s-r}(\underline{u})+\ldots+\alpha_{s}(\underline{u})}.
      \end{align*}
       Since  $\displaystyle\frac{\pi^{s-r-2}L}{\pi^{s-r-1}L}(\underline{-1})$ is straight, by Lemma \ref{iso on each block}, the set $\{\alpha_{s-r-1}(\underline{u})+\alpha_{s-r}(\underline{u})+\ldots+\alpha_{s}(\underline{u})\mid \underline{u}\in \mathbb{Z}^n\}$ is constant on $\mathcal{B}(U)$ for each subset $U$ of $\{1, \ldots, n\}$. This implies that the set $\{\alpha_{s-r-1}(\underline{u})\mid \underline{u}\in \mathbb{Z}^n\}$ is constant on $\mathcal{B}(U)$ for each subset $U$ of $\{1, \ldots, n\}$ as the sets $\{\alpha_{s-r}(\underline{u})\mid \underline{u}\in \mathbb{Z}^n\}$, $\{\alpha_{s-r+1}(\underline{u})\mid \underline{u}\in \mathbb{Z}^n\}$,\ldots,  $\{\alpha_{s}(\underline{u})\mid \underline{u}\in \mathbb{Z}^n\}$ are constant on $\mathcal{B}(U)$ for each subset $U$ of $\{1, \ldots, n\}$. This completes the proof of the theorem.
\end{proof}
\section{Proof of Theorem E: An application}
In this section, we give a proof of \hyperref[Theorem E]{Theorem E} which is the main and final result of the paper. Before proceeding further, let us again recall the hypothesis as below.
\begin{hypothesis}\label{hypo-TS-7}
Let $A$ be a Dedekind domain of characteristic zero such that for each height one prime ideal $\p$ in $A$, the local ring $A_{\p}$ has mixed characteristic with finite residue field. Suppose that $R=A[X_1,\ldots,X_n]$ is a standard $\mathbb{N}^n$-graded polynomial ring over $A$, i.e., $\deg A=\underline{0}\in \mathbb{N}^n$ and $\deg(X_j)=e_j\in \mathbb{N}^n$. Let $I$ be a $\mathfrak{C}$-monomial ideal of $R$ and let $M:= H^i_I(R)=\bigoplus_{\underline{u}\in \mathbb{Z}^n}M_{\underline{u}}$. 
\end{hypothesis}
\begin{theorem}\label{Bass numbers are constant}
    Assume the hypothesis as in  \ref{hypo-TS-7}.  Fix a  prime ideal $\p\subseteq A$ and  $i\geq 0$. Then, the set of Bass numbers $\{\mu_i(\p,M_{\underline{u}})\mid \underline{u}\in \mathbb{Z}^n\}$ is constant 
on $\mathcal{B}(U)$ for each subset $U$ of $\{1, \ldots, n\}$.
\end{theorem}
\begin{proof}
      Note that $$\mu_i(\p,M_{\underline{u}})=\mu_i(\widehat{\p A_\p},M_{\underline{u}}\otimes_A \widehat{A_\p}).$$

    Therefore we may assume that $(A,\pi A,F)$ is a complete DVR of mixed characteristic with finite residue field and it is enough to prove that for a fixed $i\geq 0$, the set $\{\mu_i(\pi A,M_{\underline{u}})\mid \underline{u}\in \mathbb{Z}^n\}$ is constant on $\mathcal{B}(U)$ for each subset $U$ of $\mathcal{S}=\{1,2\ldots,n\}$. Let $K$ be the fraction field of $A$. Since $A$ is PID, by \cite[Theorem 1.3, Part 7(a)]{TS-25} the following exact sequence splits
    $$0\rightarrow \Gamma_\pi(M_{\underline{u}})\rightarrow M_{\underline{u}}\rightarrow \overline{M_{\underline{u}}} \rightarrow 0.$$
    Hence $$M_{\underline{u}}=\Gamma_\pi(M_{\underline{u}})\oplus \overline{M_{\underline{u}}}.$$
    There are finite numbers $s, l({\underline{u}}), \alpha_j({\underline{u}})\geq 0$ (see, Setup \ref{second setup}) such that
$$\Gamma_{\pi}(M_{\underline{u}})=E_A\left(\frac{A}{\pi A}\right)^{l({\underline{u}})} \oplus\left(\bigoplus_{j=1}^{s}\left(\frac{A}{\pi^jA}\right)^{\alpha_j(\underline{u})}\right)$$ and $$\overline{M_{\underline{u}}}= \frac{M_{\underline{u}}}{\Gamma_{\pi}(M_{\underline{u}})}=A^{a(\underline{u})}\oplus K^{b(\underline{u})}.$$
Now \begin{align*}
\mu_i(\pi A,M_{\underline{u}})\ &=\mu_i\left(\pi A,\Gamma_\pi(M_{\underline{u}})\oplus \overline{M_{\underline{u}}}\right)\\ &=\dim_{A/\pi A}\Ext^i_A(A/\pi A, \Gamma_\pi(M_{\underline{u}})\oplus \overline{M_{\underline{u}}})\\ &= \dim_{A/\pi A}\Ext^i_A\left(A/\pi A, \Gamma_\pi(M_{\underline{u}})\right)+ \dim_{A/\pi A}\Ext^i_A(A/\pi A,\overline{M_{\underline{u}}})
\end{align*}
Since $A$ is a domain, the following  sequence 
\begin{equation}\label{equation for free resolution of A/pi}
0\rightarrow A\xrightarrow{\pi}A\rightarrow A/\pi A\rightarrow 0
\end{equation}gives a free resolution of $A/\pi A$. Therefore, both $\Ext^i_A\left(A/\pi A, \Gamma_\pi(M_{\underline{u}})\right)$ and $\Ext^i_A(A/\pi A,\overline{M_{\underline{u}}})$ are zero for $i\geq 2$. This implies that $\mu_i(\pi A,M_{\underline{u}})=0$ for $i\geq 2$.

Let $i=1$. Now $\Ext^1_A\left(A/\pi A, E(A/\pi  A)\right)=0$ by \cite[Proposition 3.2.12, Part (a)]{BH-93}.
Applying $\Hom(-, A/\pi^j A)$ to (\ref{equation for free resolution of A/pi}), for all $j\geq 1$ we get  $$\Ext^1_A\left(A/\pi A, A/\pi^j A\right)= A/\pi^j A/\pi  (A/\pi^jA)=A/\pi A.$$
Therefore,
\begin{align*}
   \Ext^1_A\left(A/\pi A, \Gamma_\pi(M_{\underline{u}})\right)\ &=\Ext^1_A\left(A/\pi A, E_A\left(\frac{A}{\pi A}\right)^{l({\underline{u}})} \oplus\left(\bigoplus_{j=1}^{s}\left(\frac{A}{\pi^jA}\right)^{\alpha_j(\underline{u})}\right)\right)\\ &= \Ext^1_A\left(A/\pi A, E_A\left(\frac{A}{\pi A}\right)^{l({\underline{u}})}\right) \oplus\Ext_A^1\left(A/\pi A,\bigoplus_{j=1}^{s}\left(\frac{A}{\pi^jA}\right)^{\alpha_j(\underline{u})}\right)\\&=\bigoplus_{j=1}^{s}\Ext_A^1\left(A/\pi A,\frac{A}{\pi^jA}\right)^{\alpha_j(\underline{u})}\\&=\bigoplus_{j=1}^{s}(A/\pi A)^{\alpha_j(\underline{u})}\\&=\left(A/\pi A\right)^{\sum_{j=1}^s \alpha_j(\underline{u})}.
\end{align*}
Again applying $\Hom(-, A)$ and $\Hom(-, K)$ respectively  to (\ref{equation for free resolution of A/pi}) we get, $\Ext^1_A\left(A/\pi A,  A\right)=A/\pi A$ and 
$\Ext^1_A\left(A/\pi A, K\right)=K/\pi K=0$ (as $\pi$ is invertible in $K$).
Therefore,
\begin{align*}
   \Ext^1_A\left(A/\pi A, \overline{M_{\underline{u}}}\right)\ &=\Ext^1_A\left(A/\pi A, A^{a(\underline{u})}\oplus K^{b(\underline{u})}\right)\\ &= \Ext^1_A\left(A/\pi A, A^{a(\underline{u})}\right) \oplus\Ext^1_A\left(A/\pi A, K^{b(\underline{u})}\right)\\&=(A/\pi A)^{a(\underline{u})}.
\end{align*}
Hence 
\begin{align*}
    \mu_1(\pi A,M_{\underline{u}})\ &=\dim_{A/\pi A}\Ext^1_A\left(A/\pi A, \Gamma_\pi(M_{\underline{u}})\right)+ \dim_{A/\pi A}\Ext^1_A(A/\pi A,\overline{M_{\underline{u}}})\\ &=\sum_{j=1}^s \alpha_j(\underline{u})+a(\underline{u}).
\end{align*}
Note that for all $1\leq j\leq s $, the set $\{\alpha_j(\underline{u})\mid \underline{u}\in \mathbb{Z}^n\}$ as well as the set $\{a(\underline{u})\mid \underline{u}\in \mathbb{Z}^n\}$ is constant on $\mathcal{B}(U)$ for each subset $U$ of $\{1, \ldots, n\}$ (see, Theorem \ref{proof of theorem D} and Theorem \ref{proof of theorem a}). Consequently, the set $\{\mu_1(\pi A,M_{\underline{u}})\mid \underline{u}\in \mathbb{Z}^n\}$ is constant on $\mathcal{B}(U)$.

Now we show that the set $\{\mu_0(\pi A,M_{\underline{u}})\mid \underline{u}\in \mathbb{Z}^n\}$ is constant on $\mathcal{B}(U)$.
It is easy to see that $$\Hom_A(A/\pi A, E(A/\pi A))\cong \Hom_A(A/\pi A, A/\pi^j A)\cong A/\pi A$$
 and $$\Hom_A(A/\pi A, A)=\Hom_A(A/\pi A, K)=0.$$
 Therefore,
\begin{align*}
   \Hom_A\left(A/\pi A, \Gamma_\pi(M_{\underline{u}})\right)\ &=\Hom_A\left(A/\pi A, E_A\left(\frac{A}{\pi A}\right)^{l({\underline{u}})} \oplus\left(\bigoplus_{j=1}^{s}\left(\frac{A}{\pi^jA}\right)^{\alpha_j(\underline{u})}\right)\right)\\ &= \Hom_A\left(A/\pi A, E_A\left(\frac{A}{\pi A}\right)^{l({\underline{u}})}\right) \oplus\Hom_A\left(A/\pi A,\bigoplus_{j=1}^{s}\left(\frac{A}{\pi^jA}\right)^{\alpha_j(\underline{u})}\right)\\&=(A/\pi A)^{l({\underline{u}})}\oplus\left(\bigoplus_{j=1}^{s}\Hom_A\left(A/\pi A,\frac{A}{\pi^jA}\right)^{\alpha_j(\underline{u})}\right)\\&=(A/\pi A)^{l({\underline{u}})}\oplus\left(\bigoplus_{j=1}^{s}(A/\pi A)^{\alpha_j(\underline{u})}\right)\\&=\left(A/\pi A\right)^{l(\underline{u})+\sum_{j=1}^s \alpha_j(\underline{u})}
\end{align*}
Also \begin{align*}
   \Hom_A\left(A/\pi A, \overline{M_{\underline{u}}}\right)\ &=\Hom_A\left(A/\pi A, A^{a(\underline{u})}\oplus K^{b(\underline{u})}\right)\\ &= \Hom_A\left(A/\pi A, A^{a(\underline{u})}\right) \oplus\Hom_A\left(A/\pi A, K^{b(\underline{u})}\right)\\&=0.
\end{align*}
Hence 
\begin{align*}
    \mu_0(\pi A,M_{\underline{u}})\ &=\dim_{A/\pi A}\Hom_A\left(A/\pi A, \Gamma_\pi(M_{\underline{u}})\right)+ \dim_{A/\pi A}\Hom_A(A/\pi A,\overline{M_{\underline{u}}})\\ &=l(\underline{u})+\sum_{j=1}^s \alpha_j(\underline{u})
\end{align*}

By Theorem \ref{proof of theorem D}, for all $1\leq j\leq s $, the set $\{\alpha_j(\underline{u})\mid \underline{u}\in \mathbb{Z}^n\}$ is constant on $\mathcal{B}(U)$ for each subset $U$ of $\{1, \ldots, n\}$. Also by Theorem \ref{proof of theorem b}, the set $\{l(\underline{u})\mid \underline{u}\in \mathbb{Z}^n\}$ is constant on $\mathcal{B}(U)$ so is the set $\{\mu_0(\pi A,M_{\underline{u}})\mid \underline{u}\in \mathbb{Z}^n\}$. This concludes the proof of the theorem.
\end{proof}

\bigskip

\noindent {\it Acknowledgements:} The first  author thanks the Government of India for support through the Prime Minister's Research Fellowship (PMRF ID: 1303161).

\medskip

\bibliographystyle{plain}

\end{document}